\documentclass[12pt,reqno]{amsart}
\usepackage{amssymb,amsmath,amsthm}
\usepackage{graphicx}
\usepackage{subcaption}
\usepackage{fullpage}
\usepackage{scrextend}
\usepackage{siunitx}
\usepackage{enumerate}
\usepackage{tikz,calc}
\usepackage{tikz-cd}
\usetikzlibrary{automata,positioning}
\usepackage{comment}
\usetikzlibrary{calc}
\usepackage{epsfig,graphicx}
\usepackage{listings}
\usepackage{enumerate}
\usepackage{float}
\usepackage{bbm}
\usepackage[colorlinks=true, pdfstartview=FitH, linkcolor=blue, citecolor=blue, urlcolor=blue]{hyperref}

\newcommand{\Mod}[1]{\ (\text{mod}\ #1)}
\renewcommand{\Mod}[1]{{\ifmmode\text{\rm\ (mod~$#1$)}\else\discretionary{}{}{\hbox{ }}\rm(mod~$#1$)\fi}}

\newcommand{\eps}{\varepsilon}
\newcommand{\E}{\mathbb{E}}

\newtheorem{theorem}{Theorem}
\newtheorem{thm}{Theorem}[subsection]

\newtheorem{lemma}[theorem]{Lemma}

\newtheorem{conjecture}[theorem]{Conjecture}

\theoremstyle{definition} 

\numberwithin{theorem}{section}

\title{On the Tur\'an number of the $G_{3\times 3}$ in linear hypergraphs}

 \author{J\'{o}zsef Solymosi}
\address{Department of Mathematics \\ University of British Columbia \\ Vancouver, BC, Canada \\ and Obuda University, Budapest, Hungary}
\email{solymosi@math.ubc.ca}

\begin{document}

\begin{abstract}
    We show a construction for dense 3-uniform linear hypergraphs without $3\times 3$ grids, improving the lower bound on its Tur\'an number.
    We also list some related problems.
\end{abstract}

\maketitle

\section{Introduction}
Finding the Tur\'an number of hypergraphs is challenging. The typical question asks for the maximum number of edges that can be avoided in a given hypergraph.
There are a few examples only when sharp bounds are known. Such questions are even more challenging in an important subfamily called {\em linear hypergraphs}.
In an $r$-uniform linear hypergraph, every hyperedge has $r$ vertices, and any pair of edges have at most one common vertex. 
An $r$-uniform linear hypergraph on $r^2$ vertices is called an $r$ by $r$ grid if it is isomorphic to a pattern of $r$ horizontal and $r$ vertical lines. 

Answering a question by F\"uredi and Ruszink\'o \cite{Furedi2013} about hypergraphs avoiding $3\times 3$ grids (denoted by $G_{3\times 3}$), Gishboliner and Shapira gave a construction for dense linear 3-uniform hypergraphs not containing $G_{3\times 3}$. 

\begin{figure}[H]
\centering
\includegraphics[scale=.3]{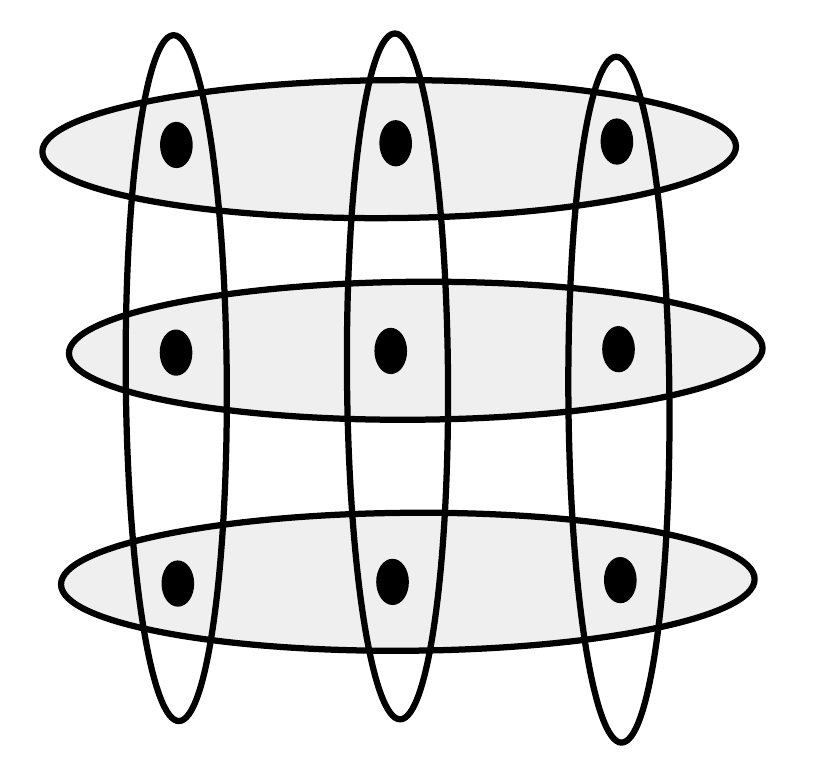}
\caption{A $G_{3\times 3}$ hypergraph on nine vertices with six edges}
\label{fig:tictac}
\end{figure}

Extremal problems about $G_{3\times 3}$ are important due to the connection to the Brown-Erd\H{o}s-S\'os conjecture, coding theory and geometric application (see examples in \cite{Furedi2013} \cite{Gishboliner2022} and \cite{Solymosi2024}). 

\begin{theorem}[Gishboliner and Shapira \cite{Gishboliner2022}]
For infinitely many $n$, there exists a linear $G_{3\times 3}$-free 3-uniform hypergraph with $n$ vertices
and $( \frac{1}{16} - o(1))n^2$ edges.
\end{theorem}

F\"uredi and Ruszink\'o conjectured that there are arbitrarily large Steiner triple systems avoiding $r\times r$ grids for any $r\ge 3$. 
They gave constructions for $r>3$, in which linear hypergraphs close to maximal density avoid $ r \times r$ grids \cite{Furedi2013}.
We are unsure about their conjecture (particularly the $r=3$ case), but at least we improve the Gishboliner-Sharipa bound.

\section{The new bound}

As in many extremal constructions, we are using objects in finite geometries. 

\begin{thm}\label{main}
For infinitely many $n$, there exists a linear $G_{3\times 3}$-free 3-uniform hypergraph with $n$ vertices
and $\left( \frac{1}{12} - o(1)\right)n^2$ edges.
\end{thm}

\begin{proof}
First, we reprove the Gishboliner-Shapira bound with the constant $\frac{1}{16}$ in a different way.
Let $q$ be a prime power, let $\mathbb{F}_q$ be the ﬁnite ﬁeld of order $q$, and let $AG(2,q)$ denote the Desarguesian aﬃne plane of order $q$.

The vertex set of the hypergraph $\mathcal{H}$ consists of the points of two parabolas, $V_1=\{(x,x^2) : x\in \mathbb{F}_q\}$ and $V_2=\{(x,x^2+1) : x\in \mathbb{F}_q\}$. 
The edge set is defined by collinear triples on the lines of the secants of $V_1$ that have a non-empty intersection with $V_2$. If the line has two intersection points in $V_2$, choose only one of them to form an edge. Let's count the number of edges. 

\medskip
There are $q$ points in $V_1$, so the number of its secants is $\binom{q}{2}$. For two distinct points $(a,a^2)$ and $(b,b^2)$ the equation of the line connecting them is $y=(a+b)x-ab.$ It intersects $V_2$ iff the discriminant $(a-b)^2-4$ is a quadratic residue or zero (we assumed here that $q$ is odd, although one could perform similar calculations for even $q$ as well). It is quadratic residue or zero for $q(q-\chi(-1))/2$ ordered $(a,b)$, $a\neq b$ pairs, where $\chi$ is the quadratic character in $\mathbb{F}_q$. 
The number of intersecting secants, and therefore the number of edges is
\[
E(\mathcal{H})=\frac{q(q-\chi(-1))}{4}\approx \frac{q^2}{4}=\left(\frac{|V(\mathcal{H}|}{2}\right)^2\frac{1}{4}=\frac{1}{16}|V(\mathcal{H}|^2
\]
for large $q$. 

\medskip
This 3-uniform linear hypergraph doesn't contain $3\times 3$ grids. If there were one in $\mathcal{H}$, then six of its vertices were in $V_1$, so by Pascal's theorem\footnote{Pascal's theorem holds over finite fields (or more precisely over projective geometries \cite{Artzy1968}).}, the three vertices in $V_2$ were collinear. On the other hand, the parabola, $V_2$, has no collinear triples.

\medskip
Most of the intersecting lines of the secants of $V_1$ intersect $V_2$ in two points, and we selected only one of them. Now we select a random subset, hoping that many lines will still intersect at a point as the number of points reduces. Let's select points of $V_2$ independently at random with probability $p$ into a set $S$. The expected size of $S$ is $pq$. The number of secants of $V_1$ with two intersection points in $V_2$ is 
$\frac{q(q-\chi(-1)-4)}{4}$, since we have to discard the tangents, the cases when the discriminant is zero, i.e. when $a-b=\pm 2$. We are looking for an asymptotic bound, so we will continue counting with $\approx q^2/4$ such lines. The expected number of lines with at least one point in $S$ is $(2p-p^2)q^2/4$. Select $p$ to maximize 
\[
\frac{(2p-p^2)q^2/4}{(q+pq)^2}=\frac{2p-p^2}{4(1+p)^2}.
\]

The maximum is achieved when $p=1/2$. Then the ratio is $1/12$ as required. As the final step, we select the edges of the random subgraph of $\mathcal{H}$. If the line of the secant has four points, choose one of the two from $S$. The secants with one point in $S$ determine a unique edge. Next, we show that one can select half of the points of $V_1$ matching the edge numbers we expect.


\begin{lemma}
 Fix any set of pairs of $N$ elements, $H\subset\binom{[N]}{2}$ with $|H|=\tfrac12\binom N2$, and let
$k=\lfloor N/2\rfloor$.
Then there is a subset $S\subset H$ containing at least one element from $$\left\lceil\frac{2kN - k^2 - k}{4}\right\rceil\approx \frac{3N^2}{16}$$ pairs in $H$.
 
\end{lemma}

Choose $S\subset[N]$ uniformly at random from all subsets of size $k$.  Define
\[
G(S)=\{\,e\in H : S\cap e\neq\emptyset\}, 
\quad
c(S)=\frac{|G(S)|}{(N+|S|)^2},
\]
where $c(S)$ is the density of edges of the hypergraph we are looking for in Theorem \ref{main}.
There is some $S$ of size $k$ with
\[
|G(S)| \;\ge\;\E\bigl[|G|\bigr].
\]
We compute this expectation:
\[
\E\bigl[|G|\bigr]
=|H|\biggl(1-\frac{\binom{N-2}{k}}{\binom{N}{k}}\biggr)
=\frac12\binom N2\biggl(1-\frac{(N-k)(N-k-1)}{N(N-1)}\biggr)
=\frac{2kN - k^2 - k}{4}.
\]

Hence
\[
\lvert G(S)\rvert \;\ge\;\left\lceil\frac{2kN - k^2 - k}{4}\right\rceil.
\]

It follows that for some \(S\) of size \(k\),
\[
c(S)
=\frac{\lvert G(S)\rvert}{(N+k)^2}
\;\ge\;
\frac{\displaystyle \left\lceil\frac{2kN - k^2 - k}{4}\right\rceil}{(N+k)^2}=\frac{1}{12} -\delta_N
\]

where the  error term
\[
\delta_N
=\frac{\bigl\lceil\frac{2kN - k^2 - k}{4}\bigr\rceil - \frac{2kN - k^2 - k}{4}}{(N+k)^2}\leq \frac{4}{9N^2}.
\]

\end{proof}

\medskip
 
 Since we need half of the points from one of the conics in a random-like manner, we can also use an algebraic, deterministic construction. For the calculations, it is easier if we take half of $V_1$, counting triples with two points in $V_2$ and at least one in $V_1$. Set $S$ to be the points above quadratic residues, i.e. $S:=\{(a^2,a^4) | a\in \mathbb{F}_p$. Then $S$ is ``random enough'' so half of the lines of point-pairs in $S$ intersect $V_2$, like in a random subset. We left the calculations to the Appendix.

\section{Further problems}

Investigating Steiner triple systems (STS-s) of size 21 with a non-trivial automorphism group, Erskine and Griggs observed that all such STS-s contain a $3\times 3$ grid \cite{ErskineGriggs2024}. It makes it plausible that every large enough STS contains a grid, contradicting the conjecture of F\"uredi and Ruszink\'o.
We state a less ambitious conjecture.

\begin{conjecture}\label{gridcore}
Every large enough STS contains a 2-core on nine vertices, where a 2-core is a hypergraph with minimum degree 2.
\end{conjecture}

It was observed by Colbourn and Fujiwara that every STS contains a core on at most ten vertices (Theorem 3. in \cite{Colbourn2009})

In a closely related problem, a stronger conjecture was stated in \cite{solymosi2017cores}: Every large enough 3-uniform linear hypergraph with $n$ vertices and $cn^2$ edges contains a core on at most nine vertices. This conjecture was motivated by the case $k = 6$ of the Brown-Erd\H{o}s-S\'os conjecture, since a 2-core on nine vertices has at least six edges\footnote{The Brown-Erd\H{o}s-S\'os conjecture states that if in a 3-uniform hypergraph no six vertices span at least nine edges (like any core on at most nine vertices), then it is sparse, it has $o(n^2)$ edges \cite{brown1973some}}. It was also investigated in \cite{Gishboliner2022} whether all dense linear hypergraphs contain a core on at most nine vertices.

There are two 2-core hypergraphs on nine vertices, which are subgraphs of any other core on nine vertices: the grid and another graph called the prism (or double triangle), which can be avoided in arbitrarily large Steiner triple systems \cite{Colbourn2009}.

\begin{figure}[h]
  \centering
\begin{tikzpicture}[scale=1.5, every node/.style={circle, draw, fill=white, inner sep=1pt}]
  \foreach \i in {0,1,2}
    \foreach \j in {0,1,2}
      \node (a\i\j) at (\j,-\i) {};

  \node at (0, 0.3) {a};
  \node at (1, 0.3) {b};
  \node at (2, 0.3) {c};
  \node at (-0.2, -1.1) {d};
  \node at (1.2, -0.9) {e};
  \node at (2.2, -1.1) {f};
  \node at (0, -2.3) {g};
  \node at (1, -2.3) {h};
  \node at (2, -2.3) {i};

  \draw (a00) -- (a01) -- (a02); 
  \draw (a00) -- (a10) -- (a20); 
  \draw (a02) -- (a12) -- (a22); 
  \draw (a20) -- (a21) -- (a22); 

  \draw[bend right=20] (a01) to (a11);
  \draw[bend right=20] (a11) to (a12);

  \draw[bend left=20] (a10) to (a11);
  \draw[bend left=20] (a11) to (a21);

\end{tikzpicture}
  \caption{The prism or double triangle.}
  \label{fig:prism}
\end{figure}

\section{Acknowledgements}
The research was supported in part by an NSERC Discovery grant and by the National Research Development and Innovation Office of Hungary, NKFIH, Grant No. KKP133819 and Excellence 151341.

\section{Appendix}

Let $S = \{(a^2,a^4)\colon a \in \mathbb{F}_p\},$ and count the number of distinct lines through pairs of points of $S$ that intersect $V_2$. The number of such lines is denoted by $N=N(p)$. $N_{two}$ and $N_{one}$ denote the lines with two and with one intersection point with $V_2$.

\medskip

Choose two distinct points
\[
P=(a^2,a^4),\quad Q=(b^2,b^4),\quad a,b\in\mathbb{F}_p,\ a\neq b.
\]

The slope of the line $PQ$ is
\[
 m = \frac{a^4 - b^4}{a^2 - b^2}
     = a^2 + b^2.
\]

Plugging $(x,y)=(a^2,a^4)$ into $y = m x + c$ gives
\[
 a^4 = (a^2+b^2)a^2 + c \,\Longrightarrow\, c = -a^2b^2,
\]

and
\[
 \ell_{a,b}:\quad y = (a^2+b^2)x - a^2b^2.
\]

The intersection points with $V_2$ satisfy
\[
(a^2+b^2)x - a^2b^2 = x^2 + 1,
\]

The discriminant is
\[
 \Delta_{a,b} = (a^2+b^2)^2 - 4(1 + a^2 b^2)
               = (a^2-b^2)^2 - 4
               = ((a-b)(a+b))^2 - 2^2.
\]

Write $u=a-b$, $v=a+b$, so
\[
 \Delta_{a,b} = (uv - 2)(uv + 2).
\]
If $\Delta_{a,b}$ is a quadratic residue in $\mathbb{F}_p$ then the line meets $V_2$ in two distinct points and $\Delta_{a,b}=0$ gives a tangent.

Let $\chi\colon\mathbb{F}_p^\times\to\{\pm1\}$ be the quadratic character, extended by $\chi(0)=0$.  A secant line $\ell_{a,b}$ with $\Delta\neq0$ meets $V_2$ in two points exactly when $\chi(\Delta_{a,b})=+1$.  Each pair $(a,b)$ and $(b,a)$ give the same line.  
\[
 N_{\rm two}(p)
 = \frac12\#\{(a,b)\colon a\neq b,\ \chi(\Delta_{a,b})=+1\}
 = \frac12\sum_{a\neq b}\frac{1+\chi(\Delta_{a,b})}{2}-\frac14 N_{one}.
\]

To evaluate the sum above, let's check
\[
 \sum_{a\neq b}\chi(\Delta_{a,b})
 = \sum_{u\neq0}\sum_{v\in\mathbb{F}_p}\chi(u^2v^2 - 4)
 = (p-1)\sum_{x\in\mathbb{F}_p}\chi(x^2 - 4),
\]
because $x=uv$ runs over all $\mathbb{F}_p$ when $v$ does.  A classical Gauss‐sum evaluation shows $\sum_x\chi(x^2-4)=-1$.  Thus
\[
 \sum_{a\neq b}\chi(\Delta_{a,b}) = -(p-1),
\]
and
\[
 N_{\rm two}(p) = \tfrac12\cdot\tfrac{(p-1)-(p-1)}{2} = 0,
\]
before tangents are added.

\subsection*{4. Tangent secants $\Delta=0$}
Tangency occurs when $\Delta_{a,b}=0$, i.e.\ $uv=\pm2$.  We count distinct lines:

-- Solve $uv=\alpha\neq0$: for each fixed $\alpha$, there are exactly $p-1$ ordered $(u,v)$ with $uv=\alpha$, hence $p-1$ unordered.  However $(u,v)$ and $(-u,-v)$ give the same $\{a,b\}$, so each unordered $\{u,v\}$ yields one line.  Thus for each $\alpha\in\{\pm2\}$ we get 1 tangent line precisely when $\alpha$ is a square in $\mathbb{F}_p$.

By quadratic reciprocity,
\[
 \chi(2)=+1 \iff p\equiv\pm1\pmod8,
 \quad
 \chi(-2)=+1 \iff p\equiv1,3\pmod8.
\]
Hence:
\[
 \#\{\text{tangent lines}\}
 = [\chi(2)=+1] + [\chi(-2)=+1]
 = \begin{cases}
 0, & p\equiv3\pmod4,\\
 2, & p\equiv1\pmod4.
 \end{cases}
\]

\subsection*{5. Final step}
Putting the two‐intersection and tangent counts together, the total number of distinct secant lines meeting $C_2$ is
\[
 N(p) = N_{\rm two}(p) + \#\{\text{tangent lines}\}
       = 0 + \begin{cases}0,& p\equiv3\pmod4,\\2,& p\equiv1\pmod4,\end{cases}
\]
i.e.
\[
 N(p)=
\begin{cases}
\displaystyle\frac{(p+1)^2}{16},&p\equiv3\pmod4,\\[0.5em]
\displaystyle\frac{(p-1)^2}{16}+2,&p\equiv1\pmod4.
\end{cases}
\]

\medskip

\end{document}